\documentclass{amsart}
\usepackage[ocgcolorlinks,linktoc=all]{hyperref}
\hypersetup{citecolor=blue,linkcolor=red}
\newtheorem{theorem}{Theorem}
\newtheorem*{thmA}{Theorem}

\newtheorem{corollary}[theorem]{Corollary}
\theoremstyle{definition}

\theoremstyle{remark}
\newtheorem{remark}[theorem]{Remark}

\numberwithin{equation}{section}

\begin{document}

\title[]
 {Stability of the Blaschke-Santal\'{o} inequality in the plane}
\author[M.N. Ivaki]{Mohammad N. Ivaki}
\address{Institut f\"{u}r Diskrete Mathematik und Geometrie, Technische Universit\"{a}t Wien,
Wiedner Hauptstr. 8--10, 1040 Wien, Austria}
\curraddr{}
\email{mohammad.ivaki@tuwien.ac.at}


\begin{abstract}
We give a stability version of of the Blaschke-Santal\'{o} inequality in the plane.
\end{abstract}

\maketitle
\section{Introduction}
The setting of this paper is the $n$-dimensional Euclidean space. A compact convex subset of $\mathbb{R}^{n}$ with non-empty interior is called a \emph{convex body}. The set of convex bodies in $\mathbb{R}^{n}$ is denoted by $\mathcal{K}^n$. Write $\mathcal{K}^n_{e}$ for the set of origin-symmetric convex bodies and $\mathcal{K}^n_{0}$ for the set of convex bodies whose interiors contain the origin.

The support function of $K\in \mathcal{K}^n$, $h_K:\mathbb{S}^{n-1}\to\mathbb{R}$, is defined by \[h_{K}(u)=\max_{x\in K}\langle x,u\rangle,\]
where $\langle \cdot,\cdot\rangle$ stands for the usual inner product of $\mathbb{R}^n.$
The polar body, $K^{\ast}$, of $K\in \mathcal{K}^n_0$ is the convex body defined by
\[K^{\ast}=\{y\in\mathbb{R}^n: \langle x,y\rangle\leq 1 \mbox{~for~all~}x\in K\}.\]
For $x\in$ int $K$, let $K^x:=(K-x)^{\ast}.$ The Santal\'{o} point of $K$, denoted by $s$, is the unique point in int $K$ such that
$$V(K^s)\leq V(K^x)$$
for all $x\in$ int $K$. For a body $K\in\mathcal{K}_e^n$, the Santal\'{o} point is at the origin. The Blaschke-Santal\'{o} inequality \cite{BW1,Santalo} states that
\begin{equation*}
V(K^s)V(K)\leq \omega_{n}^2,
\end{equation*}
with equality if and only if $K$ is an ellipsoid. Here $\omega_{n}$ is the volume of $B$, the unit ball of $\mathbb{R}^n$. The equality condition was settled by Saint Raymond \cite{SR} in the symmetric case and Petty \cite{P1}
in the general case.

A natural tool in the affine geometry of convex bodies is the Banach-Mazur distance which for two convex bodies $K,\bar{K}\in \mathcal{K}^n$ is defined by
\[d_{\mathcal{BM}}(K,\bar{K})=\min\{ \lambda\geq 1: (K-x)\subseteq \Phi(\bar{K}-y)\subseteq \lambda(K-x),~\
\Phi\in GL(n),~x,y\in\mathbb{R}^n\}.\]
It is easy to see that $d_{\mathcal{BM}}(K, \Phi \bar{K})=d_{\mathcal{BM}}(K,\bar{K})$ for all $\Phi\in GL(n).$ Moreover, the Banach-Mazur distance is multiplicative. That is,  for $K_1,K_2,K_3\in \mathcal{K}^n_e$ the following inequality holds:
\[d_{\mathcal{BM}}(K_1,K_3)\leq d_{\mathcal{BM}}(K_1,K_2)d_{\mathcal{BM}}(K_2,K_3).\]
The main result of the paper is stated in the following theorem.
\begin{thmA}
There exist constants $\gamma,~\varepsilon_0>0$, such that the following holds: If $0<\varepsilon<\varepsilon_0$ and $K$ is a convex body in $\mathbb{R}^2$ such that $V(K^s)V(K)\geq \frac{\pi^2}{1+\varepsilon},$ then $d_{\mathcal{BM}}(K,B)\leq 1+\gamma \varepsilon^{\frac{1}{4}}$. Furthermore, if $K$ is an origin-symmetric body, then $d_{\mathcal{BM}}(K,B)\leq 1+\gamma\varepsilon^{\frac{1}{2}}.$
\end{thmA}
In $\mathbb{R}^n$, $n\geq 3$, the stability of the Blaschke-Santal\'{o} inequality was first proved by K.J. B\"{o}r\"{o}czky \cite{B}, and then by K. Ball and K.J. B\"{o}r\"{o}czky \cite{BB2} with a better order of approximation (see also \cite{BBF} for the stability of functional forms of the Blaschke-Santal\'{o} inequality). In $\mathbb{R}^2$, a result has been obtained by K.J. B\"{o}r\"{o}czky and E. Makai \cite{BM} where the order of approximation in the origin-symmetric case is $1/3$ and in the general case is $1/6.$ Therefore, our main theorem provides a sharper stability result. Moreover, stability of the $p$-affine isoperimetric inequality also follows from the stability of the Blaschke-Santal\'{o} inequality (See \cite{Lutwak4,Sch} for definitions of the $p$-affine surface areas, and for the statements of the $p$-affine isoperimetric inequalities, and see also \cite{Ludiwg1,Ludiwg2} for their generalizations in the context of the Orlicz-Brunn-Minkowski theory, basic properties, and affine isoperimetric inequalities they satisfy.). Stability of the $p$-affine isoperimetric inequality, in the Hausdorff distance, for bodies in $\mathcal{K}_e^2$  was established by the author in \cite{Ivaki2} via the affine normal flow with the order of approximation equal to $3/10.$ Therefore, the main theorem here replaces $3/10$ by $1/2$ and extends that result, if $p>1$, to bodies with the Santal\'{o} points or centroids at the origin, and if $p=1$, to any convex body in $\mathcal{K}^2$. An application of such a stability result to some Monge-Amp\`{e}re functionals is given by Ghilli and Salani \cite{DP}.\newline\newline
\textbf{Acknowledgment.}
I am indebted to Monika Ludwig and the referee for the very careful reading of the original submission.
\section{Background material}
A convex body is said to be of class $\mathcal{C}^{k}_{+}$, for some $k\ge2$, if its boundary hypersurface is $k$-times continuously differentiable, in the sense of differential geometry, and the Gauss map $\nu:\partial K\to \mathbb{S}^{n-1}$, which takes $x$ on the boundary of $K$ to its unique outer unit normal vector $\nu(x)$, is well-defined and a $\mathcal{C}^{k-1}$-diffeomorphism.

Let $K,L$ be two convex bodies and $0<a<\infty$, then the Minkowski sum $K+aL$ is defined by $h_{K+aL}=h_K+ah_L$ and the mixed volume $V_1(K,L)$ ($V(K,L)$ for planar convex bodies) of $K$ and $L$ is defined by
\[V_1(K,L)=\frac{1}{n}\lim_{a\to0^{+}}\frac{V(K+aL)-V(K)}{a}.\]
A fundamental fact is that corresponding to each convex body $K$, there is a unique Borel measure $S_K$ on the unit sphere such that
\[V_1(K,L)=\frac{1}{n}\int_{\mathbb{S}^{n-1}}h_LdS_K\]
for any convex body $L$. The measure $S_K$ is called the surface area measure of $K.$

A convex body $K$ is said to have a positive continuous curvature function $f_K$, defined on the unit sphere, provided that for each convex body $L$
\[V_1(K,L)=\frac{1}{n}\int_{\mathbb{S}^{n-1}}h_Lf_Kd\sigma,\]
where $\sigma$ is the spherical Lebesgue measure on $\mathbb{S}^{n-1}.$ A convex body can have at most one curvature function; see \cite[p.~115]{bon}.
If $K$ is of class $\mathcal{C}^2_+$, then $S_K$ is absolutely continuous with respect to $\sigma$, and the Radon-Nikodym derivative $dS_K/d\sigma:\mathbb{S}^{n-1}\to\mathbb{R}$ is the reciprocal Gauss curvature of $\partial K$ (viewed as a function of the outer unit normal vectors). For every $K\in \mathcal{K}^{n},$ $V(K)=V_1(K,K).$

Of significant importance in convex geometry is the Minkowski mixed volume inequality. Minkowski's mixed volume inequality states that for $K,L\in\mathcal{K}^n,$
\[V_1(K,L)^n\geq V(K)^{n-1}V(L).\]
In the class of origin-symmetric convex bodies, equality holds if and only if $K=cL$ for some $c>0.$ In $\mathbb{R}^2$ a stronger version of Minkowski's inequality was obtained by Groemer \cite{Groemer1}. We provide his result for bodies in $\mathcal{K}_e^2:$
\begin{theorem}\cite{Groemer1}
Let $K,L\in\mathcal{K}_e^2$ and set $D(K)=2\max\limits_{\mathbb{S}^1}h_{K}$, then
\begin{align}\label{thm: groemer}
\frac{V(K,L)^2}{V(K)V(L)}-1&\geq \frac{V(K)}{4D^2(K)}\max_{u\in\mathbb{S}^1}\left|\frac{h_K(u)}{V(K)^{\frac{1}{2}}}-\frac{h_L(u)}{V(L)^{\frac{1}{2}}}\right|^2.
\end{align}
\end{theorem}
The Santal\'{o} point of $K$ is characterized by the following property
$$\int_{\mathbb{S}^{n-1}}\frac{u}{h_{K-s}^{n+1}(u)}d\sigma(u)=0.$$
Thus for an arbitrary convex body $K$, the indefinite $\sigma$-integral of $h_{K-s}^{-(n+1)}$ satisfies the sufficiency condition of Minkowski's existence theorem in $\mathbb{R}^n$ (see, for example, Schneider \cite[Theorem 8.2.2]{Sch}). Hence, there exists a unique convex body (up to translation) with curvature function
\begin{equation}\label{e: def Lambda}
f_{\Lambda K}=\frac{V(K)}{V(K^s)}h_{K-s}^{-(n+1)}.
\end{equation}
Moreover, $\Lambda \Phi K=\Phi\Lambda K$ (up to translation) for $\Phi\in GL(n)$, by \cite[Lemma 7.12]{Lutwak3}. Finally, we remark that by the Minkowski inequality for all $L\in \mathcal{K}^2$ there holds $V^2(L)=V(\Lambda L,L)^2\geq V(L)V(\Lambda L).$ Therefore $V(L)\geq V(\Lambda L)$ for all $L\in \mathcal{K}^2,$ with equality if and only if $\Lambda L$ is a translate of $L$. In this paper we always assume that the centroid of $\Lambda K$ is the origin of the plane.
\begin{remark}\label{rem: rem}
If $K\in \mathcal{K}^n$ is of class $\mathcal{C}^{\infty}_+$, then $h_K\in\mathcal{C}^{\infty}$. In fact, by definition of the class $\mathcal{C}^{\infty}_+$, the Gauss map $\nu$ is a diffeomorphism of class $\mathcal{C}^{\infty}$ and so $h_K(\cdot)=\langle \nu^{-1}(\cdot),\cdot \rangle$ is of class $\mathcal{C}^{\infty}$. In this case, since $\Lambda K$ is a solution to the Minkowski problem (\ref{e: def Lambda}) with positive $\mathcal{C}^{\infty}$ prescribed data $\frac{V(K)}{V(K^s)}h_{K-s}^{-(n+1)}$, $\Lambda K$ is of class $\mathcal{C}^{\infty}_{+};$ see Cheng and Yau \cite[Theorem 1]{ChYau}.
\end{remark}
\begin{theorem}\cite{Ivaki}\label{thm: xx} Suppose that $K\in\mathcal{K}_{e}^2$ is of class $\mathcal{C}^{\infty}_{+}$. If $m\le h_Kf_K^{1/3}\le M$ for some positive numbers $m$ and $M$, then there exist two ellipses $E_{in}$ and $E_{out}$ such that $E_{in}\subseteq K\subseteq E_{out}$ and
\[\left(\frac{V(E_{in})}{\pi}\right)^{2/3}=m,~~ \left(\frac{V(E_{out})}{\pi}\right)^{2/3}=M.\]
\end{theorem}
\begin{corollary}\label{lem: ellipsoid app}
Suppose that $K\in\mathcal{K}_{e}^2$ is of class $\mathcal{C}^{\infty}_{+}$. If $m\le h_Kf_K^{1/3}\le M$ for some positive numbers $m$ and $M$ and $V(K)=\pi$, then $m\leq 1\leq M.$
Moreover, without any assumption on the area of $K$, we have
\[d_{\mathcal{BM}}(K,B)\leq \left(\frac{M}{m}\right)^{\frac{3}{2}}.\]
\end{corollary}
\begin{proof}
Let $E_{in}$ and $E_{out}$ be the ellipses from Theorem \ref{thm: xx}. Since $V(E_{out})\geq \pi$ and $V(E_{in})\leq \pi$, the first claim follows (For another proof by Andrews, see \cite[Lemma 10]{An1} in which he does not assume that $K$ is origin-symmetric.). To prove the bound on the Banach-Mazur distance, we may first apply a special linear transformation $\Phi\in SL(2)$ such that $\Phi E_{out}$ is a disk. Then it is easy to see that $\Phi E_{out}\subseteq \frac{V(E_{out})}{V(E_{in})}\Phi E_{in}.$ Therefore
\[\Phi E_{in} \subseteq  \Phi K\subseteq \frac{V(E_{out})}{V(E_{in})}\Phi E_{in},\]
and
\[d_{\mathcal{BM}}(K,B)\leq \frac{V(E_{out})}{V(E_{in})}.\]
\end{proof}
Let $K$ be a convex body with Santal\'{o} point at the origin. In \cite{Lutwak2}, by using the affine isoperimetric inequality, Lutwak proved
\begin{equation}\label{thm: ivaki}
V(K)V(K^{\ast})\leq \omega_n^2\left(\frac{V(\Lambda K)}{V(K)}\right)^{n-1}.
\end{equation}
We will use this inequality for $n=2$ in the proof of the main theorem.
\section{Proof of the main theorem}
We shall begin by proving the claim for bodies in $\mathcal{K}_e^2$ that are of class $\mathcal{C}^{\infty}_{+}.$ By John's ellipsoid theorem, we may assume without losing any generality, after applying a $GL(2)$ transformation, that
\begin{equation}\label{ie: im}
1\leq h_K\leq \sqrt{2}.
\end{equation}
In view of inequality (\ref{thm: ivaki}), inequality $V(K)V(K^{\ast})\geq \frac{\pi^2}{1+\varepsilon}$ gives
\begin{equation}\label{ie: im3}
1\geq\frac{V(\Lambda K)}{V(K)}\geq \frac{1}{1+\varepsilon}.
\end{equation}
We will rewrite (\ref{ie: im3}) as the following equivalent expression
\[\frac{V(K,\Lambda K)^2}{V(\Lambda K)V(K)}-1\leq \varepsilon.\]
Therefore, by Groemer's stability theorem, (\ref{thm: groemer}), we obtain
\begin{align*}
\frac{V(K)}{4D^2(K)}\max_{u\in\mathbb{S}^1}\left|\frac{h_K(u)}{V(K)^{\frac{1}{2}}}-\frac{h_{\Lambda K}(u)}{V(\Lambda K)^{\frac{1}{2}}}\right|^2\leq \varepsilon.
\end{align*}
Thus for every $u\in\mathbb{S}^1$ there holds
\begin{align}\label{ie: im1}
\frac{h_K^2(u)}{V(K)}\left|\frac{V(\Lambda K)^{\frac{1}{2}}}{V(K)^{\frac{1}{2}}}-\frac{h_{\Lambda K}(u)}{h_K(u)}\right|^2\leq \frac{h_K^2(u)}{V(\Lambda K)}\left|\frac{V(\Lambda K)^{\frac{1}{2}}}{V(K)^{\frac{1}{2}}}-
\frac{h_{\Lambda K}(u)}{h_K(u)}\right|^2\leq \frac{32}{\pi} \varepsilon.
\end{align}
Using (\ref{ie: im}) we can estimate the left-hand side of (\ref{ie: im1}) to obtain
\begin{align}\label{ie: 1}
\max_{u\in\mathbb{S}^1}\left|\frac{V(\Lambda K)^{\frac{1}{2}}}{V(K)^{\frac{1}{2}}}-\frac{h_{\Lambda K}(u)}{h_K(u)}\right|^2\leq 64 \varepsilon.
\end{align}
Recall from (\ref{e: def Lambda}) that
\[h_K=\left(\frac{V(K)}{V(K^\ast)}\right)^{\frac{1}{3}}\frac{1}{f_{\Lambda K}^{\frac{1}{3}}}.\]
Plugging this into (\ref{ie: 1}) gives
\begin{align*}
\left(\frac{V(K^\ast)}{V(K)}\right)^{\frac{2}{3}}\max_{u\in\mathbb{S}^1}\left|\frac{V(\Lambda K)^{\frac{1}{2}}}{V(K)^{\frac{1}{2}}}\left(\frac{V(K)}{V(K^\ast)}\right)^{\frac{1}{3}}-
(h_{\Lambda K}f_{\Lambda K}^{\frac{1}{3}})(u)\right|^2\leq 64 \varepsilon.
\end{align*}
On the other hand, as (\ref{ie: im}) also implies $\frac{1}{\sqrt{2}}\leq h_{K^{\ast}}\leq 1$, we deduce that
\begin{align*}
\max_{u\in\mathbb{S}^1}\left|\frac{V(\Lambda K)^{\frac{1}{2}}}{V(K)^{\frac{1}{2}}}\left(\frac{V(K)}{V(K^\ast)}\right)^{\frac{1}{3}}-(h_{\Lambda K}f_{\Lambda K}^{\frac{1}{3}})(u)\right|^2\leq
(64)4^{\frac{2}{3}} \varepsilon.
\end{align*}
In particular, this last inequality leads us to
\begin{align}\label{ie: 2}
\max_{u\in\mathbb{S}^1}(h_{\Lambda K}f_{\Lambda K}^{\frac{1}{3}})(u)-\min_{u\in\mathbb{S}^1}(h_{\Lambda K}f_{\Lambda K}^{\frac{1}{3}})(u)\leq 2^{\frac{25}{6}} \varepsilon^{\frac{1}{2}}.
\end{align}
By multiplying $\Lambda K$ with $\sqrt{\frac{\pi}{V(\Lambda K)}}$ we have $V\left(\sqrt{\frac{\pi}{V(\Lambda K)}}\Lambda K\right)=\pi$.
So by Remark \ref{rem: rem}, Corollary \ref{lem: ellipsoid app}, and (\ref{ie: 2}) we get
\[2^{\frac{25}{6}}\varepsilon^{\frac{1}{2}} \left(\frac{\pi}{V(\Lambda K)}\right)^{2/3}+1\geq \left(\frac{\pi}{V(\Lambda K)}\right)^{2/3}\max_{\mathbb{S}^1}(h_{\Lambda K}f_{\Lambda K}^{\frac{1}{3}}),\]
and
\[1-2^{\frac{25}{6}}\varepsilon^{\frac{1}{2}} \left(\frac{\pi}{V(\Lambda K)}\right)^{2/3}\leq \left(\frac{\pi}{V(\Lambda K)}\right)^{2/3}\min_{\mathbb{S}^1}(h_{\Lambda K}f_{\Lambda K}^{\frac{1}{3}}).\]
Furthermore, notice that by (\ref{ie: im}) and (\ref{ie: im3}) the following inequality holds:
\[1-2^{\frac{25}{6}}\varepsilon^{\frac{1}{2}} \left(\frac{\pi}{V(\Lambda K)}\right)^{2/3}\geq 1-2^{\frac{25}{6}}\varepsilon^{\frac{1}{2}}\left(1+\varepsilon\right)^{\frac{2}{3}}. \]
Take $\varepsilon$ small enough such that
\[1-2^{\frac{25}{6}}\varepsilon^{\frac{1}{2}}\left(1+\varepsilon\right)^{2/3}>0.\]
So far we have proved: If $\varepsilon$ is small enough, then
 \[\max_{\mathbb{S}^1}(h_{\Lambda K}f_{\Lambda K}^{\frac{1}{3}})\leq \left(1+2^{\frac{25}{6}}\varepsilon^{\frac{1}{2}}\left(1+\varepsilon\right)^{2/3} \right)\left(\frac{\pi}{V(\Lambda K)}\right)^{-2/3},\]
and
\[\min_{\mathbb{S}^1}(h_{\Lambda K}f_{\Lambda K}^{\frac{1}{3}})\geq \left(1-2^{\frac{25}{6}}\varepsilon^{\frac{1}{2}}\left(1+\varepsilon\right)^{2/3} \right)\left(\frac{\pi}{V(\Lambda K)}\right)^{-2/3}>0.\]
With the aid of these last inequalities and Corollary \ref{lem: ellipsoid app} we deduce that
\begin{equation}\label{e: bm1}
d_{\mathcal{BM}}(\Lambda K,B)\leq \left(\frac{1+2^{\frac{25}{6}}\varepsilon^{\frac{1}{2}}\left(1+\varepsilon\right)^{2/3}}{1-2^{\frac{25}{6}}\varepsilon^{\frac{1}{2}}\left(1+\varepsilon\right)^{2/3}}\right)^{3/2}.
\end{equation}
We return to inequality (\ref{ie: 1}) and combine it with (\ref{ie: im3}) to get
\begin{align*}
-8\varepsilon^{\frac{1}{2}}+\frac{1}{(1+\varepsilon)^{\frac{1}{2}}}\leq
-8\varepsilon^{\frac{1}{2}}+\frac{V(\Lambda K)^{\frac{1}{2}}}{V(K)^{\frac{1}{2}}}\leq \frac{h_{\Lambda K}}{h_K}\leq 8\varepsilon^{\frac{1}{2}}+\frac{V(\Lambda K)^{\frac{1}{2}}}{V(K)^{\frac{1}{2}}}\leq
1+8\varepsilon^{\frac{1}{2}}.
\end{align*}
Furthermore, take $\varepsilon$ small enough such that $-8\varepsilon^{\frac{1}{2}}+\frac{1}{(1+\varepsilon)^{\frac{1}{2}}}>0.$ Consequently
\begin{equation}\label{e: bm2}
d_{\mathcal{BM}}(K,\Lambda K)\leq \frac{1+8\varepsilon^{\frac{1}{2}}}{-8\varepsilon^{\frac{1}{2}}+\frac{1}{(1+\varepsilon)^{\frac{1}{2}}}}.
\end{equation}
Taking into account (\ref{e: bm1}), (\ref{e: bm2}), and the multiplicativity of the Banach-Mazur distance results in the desired estimate:
\[d_{\mathcal{BM}}(K,B)\leq \left(\frac{1+2^{\frac{25}{6}}\varepsilon^{\frac{1}{2}}\left(1+\varepsilon\right)^{2/3}}{1-2^{\frac{25}{6}}\varepsilon^{\frac{1}{2}}\left(1+\varepsilon\right)^{2/3}}\right)^{3/2}
\left(\frac{1+8\varepsilon^{\frac{1}{2}}}{-8\varepsilon^{\frac{1}{2}}+\frac{1}{(1+\varepsilon)^{\frac{1}{2}}}}\right)\leq 1+\gamma\varepsilon ^{\frac{1}{2}},\]
for some universal $\gamma>0$, provided that $\varepsilon$ is small enough.

It follows from \cite[Section 3.4]{Sch} that the class of $\mathcal{C}^{\infty}_{+}$ origin-symmetric convex bodies is dense in $\mathcal{K}^n_e$. Therefore, an approximation argument will prove that the claim of the main theorem, in fact, holds for any origin-symmetric convex body. To get the more general result, for bodies in $\mathcal{K}^2$, we will first need to recall Theorem 1.4 of B\"{o}r\"{o}czky from \cite{B} and a theorem of Meyer and Pajor from \cite{MP}:
\begin{thmA}[B\"{o}r\"{o}czky, \cite{B}]
For any convex body K in $\mathbb{R}^n$ with $d_{\mathcal{BM}}(K,B)\geq 1+\varepsilon$ for $\varepsilon >0$, there exists an origin-symmetric convex body $C$ and a constant $\gamma'>0$ depending on $n$ such that $d_{\mathcal{BM}}(C,B)\geq 1+\gamma'\varepsilon^2$ and $C$ results from $K$ as a limit of subsequent Steiner symmetrizations
and affine transformations.
\end{thmA}
\begin{thmA}[Meyer, Pajor, \cite{MP}]
Let K be a convex body in $\mathbb{R}^n$, $H$ be a hyperplane, and let $K_H$ be the Steiner symmetral of $K$ with respect to $H.$ If $s$ and $s'$
denote the Santal\'{o} points of $K$ and $K_H$, respectively, then $s'\in H$, and $V(K^s)\leq V((K_H )^{s'}).$
\end{thmA}
Now we give the proof in the general case by contraposition. Let $K$ be a convex body such that \[d_{\mathcal{BM}}(K,B)>1+\left(\frac{\gamma}{\gamma'}\right)^{\frac{1}{2}}\varepsilon^{\frac{1}{4}},\]
where $\gamma'$ is the constant in B\"{o}r\"{o}czky's theorem. So by the last two theorems, there exists an origin-symmetric convex body $C$, such that $V(C)V(C^{\ast})\geq V(K)V(K^{s})$ and $d_{\mathcal{BM}}(C,B)>
1+\gamma\varepsilon^{\frac{1}{2}}.$ Moreover, $d_{\mathcal{BM}}(C,B)> 1+\gamma\varepsilon^{\frac{1}{2}}$ implies that
\[V(C)V(C^{\ast})< \frac{\pi^2}{1+\varepsilon}.\]
Therefore
\[V(K)V(K^{s})< \frac{\pi^2}{1+\varepsilon}.\]
The argument is complete.

\bibliographystyle{amsplain}

\end{document}